\documentclass[a4paper, 11pt, reqno]{amsart}
\usepackage[english]{babel}
\usepackage[utf8]{inputenc}
\usepackage{amsthm, amsmath, amssymb, amsrefs, enumerate, enumitem, mathtools, bbm, mathrsfs, yhmath, microtype}

\usepackage[top=1.45in, bottom=1.45in, left=1.25in, right=1.25in]{geometry}
\binoppenalty=\maxdimen
\relpenalty=\maxdimen

\usepackage{lmodern}

\usepackage{hyperref}

\theoremstyle{plain}
\newtheorem{thm}{Theorem}[section]
\newtheorem{prop}[thm]{Proposition}
\newtheorem{lemma}[thm]{Lemma}

\theoremstyle{definition}
\newtheorem{defn}[thm]{Definition}

\theoremstyle{remark}
\newtheorem*{rmk}{Remark}
\newtheorem*{rmks}{Remarks}

\makeatletter
\let\@@pmod\pmod
\DeclareRobustCommand{\pmod}{\@ifstar\@pmods\@@pmod}
\def\@pmods#1{\mkern4mu({\operator@font mod}\mkern 6mu#1)}
\makeatother

\numberwithin{equation}{section}

\newcommand{\C}{\mathbb{C}}
\renewcommand{\H}{\mathbb{H}}
\newcommand{\Z}{\mathbb{Z}}
\newcommand{\Q}{\mathbb{Q}}
\newcommand{\N}{\mathbb{N}}
\newcommand{\R}{\mathbb{R}}

\newcommand{\slz}{{\text {\rm SL}}_2(\mathbb{Z})}
\newcommand{\gaminf}{_{\Gamma_{\infty}} \backslash ^{\Gamma}}
\DeclareMathOperator{\sgn}{sgn}
\DeclareMathOperator{\re}{Re}
\DeclareMathOperator{\im}{Im}
\newcommand{\vt}[1]{\left\lvert #1 \right\rvert}
\newcommand{\Qc}{\mathcal{Q}}
\newcommand{\Ec}{\mathcal{E}}
\newcommand{\geo}{\Gamma_Q \backslash S_Q}
\newcommand{\Cc}{\mathcal{C}}
\newcommand{\Pb}{\mathbb{P}}
\newcommand{\Dc}{\mathcal{D}}
\newcommand{\Fc}{\mathcal{F}}
\newcommand{\Nc}{\mathcal{N}}
\newcommand{\Id}{\mathbbm{1}}
\newcommand{\Ecwh}{\widehat{\mathcal{E}}}
\newcommand{\Ecwt}{\widetilde{\mathcal{E}}}

\title[Locally harmonic Maa{\ss} forms]{Locally harmonic Maaß forms of positive even weight}
\author{Andreas Mono}
\address{Department of Mathematics and Computer Science, Division of Mathematics, University of Cologne, Weyertal 86-90, 50931 Cologne, Germany}
\email{amono@math.uni-koeln.de}

\begin{document}

\begin{abstract}
We twist Zagier's function $f_{k,D}$ by a sign function and a genus character. Assuming weight $0 < k \equiv 2 \pmod*{4}$, and letting $D$ be a positive non-square discriminant, we prove that the obstruction to modularity caused by the sign function can be corrected obtaining a locally harmonic Maa{\ss} form or a local cusp form of the same weight. In addition, we provide an alternative representation of our new function in terms of a twisted trace of modular cycle integrals of a Poincar{\'e} series due to Petersson.
\end{abstract}



\maketitle

\section{Introduction and statement of results}
In $1975$, Zagier \cite{zagier75} defined the function
\begin{align*}
f_{k,D}(\tau) \coloneqq \sum_{Q \in \Qc(D)} \frac{1}{Q(\tau,1)^{k}}, \qquad \tau \in \H \coloneqq \left\{z \in \C \colon \im(z) > 0\right\},
\end{align*}
to investigate the Doi$-$Naganuma lift. Here and troughout, $\Qc(D)$ is the set of all integral binary quadratic forms of discriminant $D \in \Z$, and $k \geq 2$. On one hand, if $D > 0$, Zagier proved that they define holomorphic cusp forms of weight $2k$ for $\Gamma \coloneqq \slz$, and computed their Fourier expansions. On the other hand, if $D < 0$, Bengoechea \cite{beng13} proved that these are meromorphic cusp forms with respect to the same data, namely meromorphic modular forms which decay like cusp forms towards $i\infty$. The poles are precisely the CM points (sometimes called Heegner points instead) of discriminant $D$, and of order $k$.

Parson \cite{parson}*{Theorem 3.1} investigated Zagier's $f_{k,D}$ function based on an individual equivalence class $[Q]_{\sim} \in \Qc(D)\slash\Gamma$ of indefinite integral binary quadratic forms, and twisted it by a sign function. More precisely, she defined
\begin{align*}
f_{k,Q}(\tau) \coloneqq \sum_{\widehat{Q} \sim Q} \frac{\sgn{(\widehat{Q})}}{\widehat{Q}(\tau,1)^k}, \qquad \sgn(Q) = \sgn{\left([a,b,c]\right)} \coloneqq \begin{cases}
\sgn(a) & \text{ if } a \neq 0, \\
\sgn(c) & \text{ if } a = 0.
\end{cases}
\end{align*}
Due to the presence of the sign function, we have a non-zero error to modularity, which is a finite sum, and explicitly given by
\begin{align*}
F_{k,Q}(\tau) \coloneqq f_{k,Q}(\tau) - \tau^{-2k}f_{k,Q}\left(-\frac{1}{\tau}\right) = 2\sum_{\substack{[a,b,c]=\widehat{Q} \sim Q \\ \sgn(ac) < 0}} \frac{\sgn{(\widehat{Q})}}{\widehat{Q}(\tau,1)^k}.
\end{align*}
In other words, the function $f_{k,Q}$ is a modular integral of weight $2k$ for the rational period function $F_{k,Q}(\tau)$. We refer the reader to the work of Knopp \cite{knopp} for more details.

In a recent article \cite{hypeis1}, the author investigated a certain class of Eisenstein series
\begin{align}
\Ec_{k,D}(\tau, s) &\coloneqq \sum_{0 \neq Q \in \Qc(D)\slash\Gamma} \chi_d\left(Q\right) \sum_{\widehat{Q} \sim Q} \frac{\sgn{(\widehat{Q})}^{\frac{k}{2}} \im(\tau)^s}{\widehat{Q}(\tau,1)^{\frac{k}{2}} \vt{\widehat{Q}(\tau,1)}^s}, \label{eq:eisold}
\end{align} 
for any $k \in 2\N$, which arises by applying Hecke's trick to Parson's construction. The function $\chi_d$ is a genus character (defined in Section \ref{sec:prel}). By results of Petersson \cite{pet48}*{Satz 1, Satz 4, Satz 6}, the sum converges absolutely for any $s \in \C$ with $\re(s) > 1-\frac{k}{2}$. Like in the case of $f_{k,D}$, the behaviour of $\Ec_{k,D}(\tau, s)$ is dictated by the sign of $D$, and consequently we distinguish between hyperbolic ($D > 0$), parabolic ($D = 0$), and elliptic ($D < 0$) Eisenstein series. This terminology comes from the fact that one can associate a quadratic form to any $\gamma \in \Gamma \setminus \left\{\pm \Id\right\}$\footnote{Explicitly given by $Q_{\gamma}(x,y) \coloneqq cx^2+(d-a)xy-by^2$ for $\gamma = \left(\begin{smallmatrix} a & b \\ c & d \end{smallmatrix}\right) \in \Gamma$.}, and the sign of its discriminant depends precisely on hyperbolicity, parabolicity, or ellipticity of $\gamma$. Although we focus on the case of weights $k \in 2\N$, one may also consider different weights. For instance, all three types of Eisenstein series were studied by Jorgenson, Kramer, von Pippich, Schwagenscheidt, V{\"o}lz for weight $k=0$, see \cite{jokrvp10}*{Theorem 4.2}, \cite{pi16}*{Section 4}, \cite{pischvoe17}*{Theorem 1.2}.

The paper \cite{hypeis1} as well as the present one are devoted to the hyperbolic case. Letting $D > 0$ be a non-square discriminant, and $d$ be a positive fundamental discriminant dividing $D$, we computed the Fourier expansion of hyperbolic Eisenstein series for any integral weight $k \in 2\N$ at $s=0$ to prove a conjecture of Matsusaka \cite{matsu2}*{eq. (2.12)} about their analytic continuation in weight $2$. This computation extends earlier work by Gross, Kohnen, Zagier \cite{grokoza}*{p.\ 517}, who dealt with weights $4 | k > 2$ not involving the sign function. In turn, the computation for weights $k \in 2\N$ relies mainly on results of Duke, Imamo\={g}lu, T\'{o}th \cite{duimto11} after appealing to Zagiers work \cite{zagier75}*{Appendix 2} on the Fourier expansion of his aforementioned function. Furthermore, we computed the analytic continuation $\Ec_{2,D}(\tau,0)$ explicitly. Up to the addition of the completed Eisenstein series $E_2^*$ and some constants, it agrees with another modular integral with rational period function, which was studied by Duke, Imamo\={g}lu, T\'{o}th in \cite{duimto10}.

In addition, one can inspect the automorphic object arising from the analytic continuation to $s = 0$. On one hand, the parabolic and elliptic (twisted) Eisenstein series extend to an ordinary and a polar harmonic Maa{\ss} form respectively in weight $2$. (We define all occuring types of Maa{\ss} forms in Section \ref{sec:prel}.) While the parabolic case is known by Roelcke \cite{roe67} and Selberg \cite{sel56}, the elliptic case was proven by Matsusaka in \cite{matsu2}*{Theorem 2.3} by combining results of Bringmann, Kane \cite{brika16} and of Bringmann, Kane, Löbrich, Ono, Rolen \cite{brikaloeonro}.  On the other hand, the hyperbolic Eisenstein series $\Ec_{2,D}(\tau,0)$ (with $D$, $d$ as above) coincides with a locally harmonic Maa{\ss} form for any $\tau$ with sufficiently large imaginary part. This raises the natural question towards the obstruction of $\Ec_{k,D}(\tau,s)$ to coincide with a local automorphic form, whenever the imaginary part of $\tau$ is not sufficiently large. To this end, we relate $\Ec_{k,D}(\tau,s)$ to the completed hyperbolic Eisenstein series
\begin{align}
\Ecwh_{k,D}(\tau, s) &\coloneqq \sum_{0 \neq Q \in \Qc(D)\slash\Gamma} \chi_d\left(Q\right) \sum_{\widehat{Q} \sim Q} \frac{\sgn{(\widehat{Q}_{\tau})}^{\frac{k}{2}} \im(\tau)^s}{\widehat{Q}(\tau,1)^{\frac{k}{2}} \vt{\widehat{Q}(\tau,1)}^s}, \label{eq:eisnew} \\
Q_{\tau} &= [a,b,c]_{\tau} \coloneqq \frac{1}{\im(\tau)}\left(a\vt{\tau}^2+b\re(\tau)+c\right), \nonumber
\end{align}
outside the net of Heegner geodesics 
\begin{align*}
\Nc(D) \coloneqq \bigcup_{[a,b,c]=Q \in \Qc(D)} \left\{\tau \in \H \ \colon a\vt{\tau}^2+bu+c=0 \right\}
\end{align*}
by adding a correction term to $\Ec_{k,D}(\tau)$ (see equation \eqref{eq:eisidentity}). A possible connection of our correction term to quantum modular forms (introduced by Zagier \cite{zagier10}) is stated in Section \ref{sec:hypeis}.

In particular, the function $\Ecwh_{k,D}(\tau, s)$ is modular of weight $k$ outside $\Nc(D)$. To describe the result, we let $\Cc_{\kappa}(h,Q)$ be the weight $\kappa$ cycle integral of $h$ associated to $Q$ (defined in equation \eqref{eq:cycdef}). Moreover, we let $\Pb_k(z_1,z_2)$ be a Poincar{\'e} series due to Petersson \cite{pet48} (see Definition \ref{def:poincare}), whose properties are collected in Lemma \ref{lem:poincare} below. We refer the reader to Subsection \ref{sec:lhmf} for definitions of our local automorphic forms.
\begin{thm}\label{thm:main}
Let $0 < k \equiv 2 \pmod*{4}$, and $\tau \in \H \setminus \Nc(D)$. Let $D > 0$ be a non-square discriminant, and $d$ be a positive fundamental discriminant dividing $D$.  
\begin{enumerate}[label=(\roman*), leftmargin=*, labelindent=0pt]
\item The function $\Ecwh_{2,D}(\tau,0)$ is a locally harmonic Maa{\ss} form of weight $2$ for $\Gamma$ with exceptional set $\Nc(D)$ as a function of $\tau$. 
\item If $2 < k \equiv 2 \pmod{4}$ then $\Ecwh_{k,D}(\tau,0)$ is a local cusp form of weight $k$ for $\Gamma$ with exceptional set $\Nc(D)$ as a function of $\tau$. 
\item Moreover, we have the alternative representation
\begin{align*}
\Ecwh_{k,D}(\tau, 0) = \sum_{Q \in \Qc(D)\slash\Gamma} \chi_d(Q) \begin{cases}
\frac{-2}{D}\Cc_0\left(\frac{j'(\tau)}{j(\cdot)-j(\tau)} - E_2^*(\tau),Q\right) & \text{ if } k = 2, \\
C(k)\Cc_{2-k}\left(\Pb_k(\tau,\cdot),Q\right) & \text{ if } k > 2,
\end{cases}
\end{align*}
where $C(k)$ is an explicit constant provided in equation \eqref{eq:finalconstant}.
\end{enumerate}
\end{thm}

\begin{rmks}
\
\begin{enumerate}[labelindent=0pt, leftmargin=*]
\item The cycle integral $\Cc_k\left(\Pb_k(\cdot,\tau),Q\right)$ was computed by Löbrich, Schwagenscheidt \cite{loeschw}. Let $Q_0 \in \Qc(D)$, and $\Fc_{1-k,Q_0}$ be the locally harmonic Maa{\ss} form introduced by Bringmann, Kane, Kohnen \cite{brikako} (see Section \ref{sec:lhmf}) with summation restricted to quadratic forms equivalent to $Q_0$ under $\Gamma$. Then \cite{loeschw}*{Theorem 4.2} states that
\begin{align*}
\Fc_{1-\frac{k}{2},Q_0}(\tau) = \frac{D^{-\frac{k}{4}}}{2\pi} \Cc_k\left(\Pb_k(\cdot,\tau),Q_0\right).
\end{align*}
In other words, a cycle integral of $\Pb_k$ in either of its arguments yields a local automorphic form in the other argument.
\item A natural splitting of $z_2 \mapsto \Pb_k(z_1,z_2)$ into meromorphic and non-meromorphic parts is due to Bringmann, Kane \cite{brika20}*{equation (3.6)}.
\item Choosing $d=1$, the function $\Ecwh_{2\kappa+2,D}(\tau, 0)$, $\kappa \in 2\N$, also appears in a slightly different normalization in \cite{brimo}*{(1.7)}, and further properties of it are stated in \cite{brimo}*{Section $4$}. In particular, $\Ecwh_{2\kappa+2,D}(\tau, 0)$ gives rise to a locally harmonic Maa{\ss} form of weight $-2\kappa$, whose properties are discussed in \cite{brimo}*{Theorem 1.2}.
\end{enumerate}
\end{rmks}

As an application of Theorem \ref{thm:main}, we would like to highlight a possible connection to twisted central $L$-values. This goes back to Kohnen \cite{koh}*{Proposition 7, Corollary 3}, who established an identity between the Petersson inner product of a cusp form with Zagiers $f_{k,D}$ function, and such $L$-values for positive even weights. More recently, Kohnen's work was utilized by Ehlen, Guerzhoy, Kane, Rolen \cite{ehgukaro}*{Theorem 1.1} to prove a criterion on the vanishing of certain twisted $L$-values under some technical conditions. Their argument relies on the theory of locally harmonic Maa{\ss} forms, and in particular on the connection of the $f_{k,D}$ function to the locally harmonic Maa{\ss} form $\Fc_{1-k,D}$, see Section \ref{sec:lhmf}. (In addition, the theory of theta lifts comes in handy to ensure existence of an analytic continuation of $\Fc_{1-k,D}$ to the case $k=1$.) Being more precise, the form $\Fc_{1-k,D}$ splits into three components (cf.\ \cite{brikako}*{Theorem 7.1}). Two of them are a holomorphic and a non-holomorphic Eichler integral of the $f_{k,D}$-function, while the third component is a so called local polynomial, which captures the behaviour of $\Fc_{1-k,D}$ between different connected components of $\H \setminus \Nc(D)$. The idea of the paper \cite{ehgukaro} now is to formulate a condition on the local polynomial of $\Fc_{1-k,D}$, evaluated at two special points on the real axis, and relate this conditions to the twisted central $L$-values via the work of Kohnen and of Bringmann, Kane, Kohnen. 

Since the function $\Ecwh_{k,D}(\tau,0)$ is a twisted version of the function $f_{\frac{k}{2},D}$, and since we found a connection of $\Ecwh_{k,D}(\tau,0)$ to a locally harmonic Maa{\ss} form (resp. local cusp form), we expect that $\Ecwh_{k,D}(\tau,0)$ may serve as a ``building block'' to detect the vanishing of suitable twisted $L$-values as well. This inspection is motivated by our remarks following Theorem \ref{thm:main}.
$\\$
$\\$
\indent The paper is organized as follows. We summarize the framework of this paper in Section \ref{sec:prel}. Section \ref{sec:hypeis} is devoted to the properties of hyperbolic Eisenstein series. This enables us to prove Theorem \ref{thm:main} in Section \ref{sec:proofmain}.

\section*{Acknowledgements:}
The author would like to thank his PhD-advisor Kathrin Bringmann for her continuous valuable feedback to this work. In addition, the author would like to thank Markus Schwagenscheidt and Joshua Males for useful conversations on the topic, and Toshiki Matsusaka as well as the anonymous referee for helpful comments on an earlier version.

\section{Preliminaries} \label{sec:prel}
We let $\tau = u+iv$, and $q \coloneqq \mathrm{e}^{2 \pi i \tau}$ troughout.

\subsection{Integral binary quadratic forms}
Let $Q$ be an integral binary quadratic form, and we abbreviate such forms by the terminology ``quadratic form'' throughout. We call a quadratic form primitive if its coefficients are coprime. The full modular group $\Gamma$ acts on the set of quadratic forms by letting
\begin{align*}
\left(Q \circ \left(\begin{smallmatrix} a & b \\ c & d \end{smallmatrix}\right)\right)(x,y) &\coloneqq Q(ax+by, cx+dy),
\end{align*}
and this action induces an equivalence relation, which we denote by $\sim$. Moreover, the action of $\Gamma$ on $\H$ by fractional linear transformations is compatible with the action of $\Gamma$ on the set of quadratic forms, in the sense that
\begin{align*}
\left(Q \circ \left(\begin{smallmatrix} a & b \\ c & d \end{smallmatrix}\right)\right)(\tau,1) = (c\tau+d)^2 Q(\gamma\tau,1).
\end{align*}
A quadratic form $Q$ may be written as $[a,b,c]$, and we denote its discriminant by
\begin{align*}
\Dc([a,b,c]) \coloneqq b^2-4ac \in \Z.
\end{align*}
One can check that equivalent quadratic forms have the same discriminant. The set $\Qc(D)\slash\Gamma$ is finite, whenever $D \neq 0$, and its cardinality is called the class number $h(D)$. If $D \equiv 0 \pmod*{4}$ or $D \equiv 1 \pmod*{4}$, then $\Qc(D)\slash\Gamma$ is non-empty. To simplify notation, we identify an equivalence class in $\Qc(D)\slash\Gamma$ with any representative of it throughout. A good reference on this is Zagier's book \cite{zagier81}.

\subsection{Genus characters}
We follow the exposition given by Gross, Kohnen, Zagier in \cite{grokoza}*{p.\ 508}. Let $Q = [a,b,c]$ be a quadratic form, and observe that $\gcd{(a,b,c)}$ is invariant under $\sim$ as well. For any $D \neq 0$, let $d$ be a fundamental discriminant dividing $D$, and stipulate $d=0$ if $D=0$. We say that an integer $n$ is represented by $Q$ if there exist $x$, $y \in \Z$, such that $Q(x,y) = n$, and recall the the Kronecker symbol $\big(\frac{d}{\cdot}\big)$. This established, an extended genus character associated to $D$ is given by
\begin{align*}
\chi_d\left([a,b,c]\right) &\coloneqq \begin{cases}
\left(\frac{d}{n}\right) & \text{ if } \gcd{(a,b,c,d)} = 1, [a,b,c] \text{ represents } n, \gcd{(d,n)} = 1, \\
0 & \text{ if } \gcd{(a,b,c,d)} > 1.
\end{cases}
\end{align*}
One can check that such an integer $n$ always exists, and that the definition is independent from its choice. Since equivalent quadratic forms represent the same integers, a genus character descends to $\Qc(D)\slash\Gamma$. If $d=1$, the character is trivial, and if $d=0$, we have $\chi_0(Q) = 0$ except $Q$ is primitive, and represents $\pm 1$. In the latter case, we note that such a quadratic form is equivalent to either $[-1,0,0]$ or $[1,0,0]$. Lastly, we have
\begin{align*}
\chi_d(-Q) = \sgn(d)\chi_d(Q)
\end{align*}
for every $d \neq 0$, linking the two choices $\pm d$. We refer the reader to \cite{grokoza}*{Proposition 1 and 2} regarding additional properties of $\chi_d$.

\subsection{Heegner geodesics}
Once more, let $Q=[a,b,c]$, and suppose that $\Dc(Q) > 0$. Hence, $Q$ is indefinite, and $Q(\tau,1)$ has the two distinct zeros 
\begin{align*}
\frac{-b - \Dc(Q)^{\frac{1}{2}}}{2a}, \quad \frac{-b + \Dc(Q)^{\frac{1}{2}}}{2a} \in \R \cup \{\infty\}.
\end{align*}
If $a=0$, then the second zero is given by $-\frac{c}{b}$. We associate to $Q$ the Heegner geodesic
\begin{align*}
S_Q \coloneqq \left\{\tau \in \H \ \colon a\vt{\tau}^2+bu+c=0 \right\},
\end{align*}
which connects the two zeros of $Q(\tau,1)$. On one hand, if $\Dc(Q)$ is a square and $a \neq 0$, then both zeros are rational. In other words, one zero of $Q(\tau,1)$ is $\Gamma$-equivalent to $\infty$, and $S_Q$ is a straight line in $\H$, perpendicular to $\R$, based on the second zero. Moreover, the stabilizer
\begin{align*}
\Gamma_Q \coloneqq \left\{\gamma \in \Gamma \colon Q \circ \gamma = Q\right\}
\end{align*}
is trivial in this case. On the other hand, if $\Dc(Q) > 0$ is not a square and $a \neq 0$, then both zeros of $Q(\tau,1)$ are real quadratic irrationals, which are Galois conjugate to each other. The geodesic $S_Q$ is an arc in $\H$, which is perpendicular to $\R$, and $S_Q$ is preserved by $\Gamma_Q$.

We stipulate that $D$ is a positive non-square discriminant. We obtain infinitely many connected components on $\H$, and finitely many such components in a fundamental domain for $\Gamma$, because the class number of $D$ is finite. Since $D$ is not a square, each geodesic $S_Q$ divides $\H$ into a bounded and an unbounded component, and we denote the bounded component (``interior'') of $\H\backslash S_Q$ by $A_Q$. Moreover, there is precisely one unbounded connected component in a fundamental domain for $\Gamma$, to which we refer as the region ``above'' the net of geodesics. 

Furthermore, we introduce the characteristic funtion
\begin{align*}
\Id_Q(\tau) \coloneqq \begin{cases}
1 & \text{ if } \tau \in A_Q, \\
0 & \text{ if } \tau \not\in A_Q,
\end{cases}
\end{align*}
whenever $\tau \in \H \setminus \Nc(D)$. Variants of $\Id_Q$ appear in \cite{schw18}*{Corollary 5.3.5}, and in \cite{matsu1}*{p.\ 8}. 

We collect the properties of our sign functions.
\begin{lemma} \label{lem:signs}
\
\begin{enumerate}[label=(\roman*), labelindent=0pt, leftmargin=*]
\item For every $\gamma \in \Gamma$, we have
\begin{align*}
Q_{\gamma \tau} = \left(Q \circ \gamma\right)_{\tau}.
\end{align*}
\item We have that $\tau \in A_Q$ if and only if 
\begin{align*}
\sgn(Q)\sgn\left(Q_{\tau}\right) < 0.
\end{align*}
\item If $\tau \in \H \setminus \Nc(D)$, then the sign functions $\sgn(Q)$, $\sgn\left(Q_{\tau}\right)$, and $\Id_Q(\tau)$ are related by
\begin{align*}
\sgn\left(Q_{\tau}\right) = \sgn(Q) \left(1-2\Id_Q(\tau)\right).
\end{align*}
\end{enumerate}
\end{lemma}

\begin{proof}
It suffices to check the first item for the two generators
\begin{align*}
S \coloneqq \left(\begin{array}{cc} 0 & -1 \\ 1 & 0 \end{array}\right), \qquad T \coloneqq \left(\begin{array}{cc} 1 & 1 \\ 0 & 1 \end{array}\right)
\end{align*}
of $\Gamma$. Indeed, we calculate that
\begin{align*}
Q_{S\tau} &= \frac{a\vt{S\tau}^2+b\re\left(S\tau\right)+c}{\im\left(S\tau\right)} = \frac{a\vt{-\frac{1}{\tau}}^2-b\frac{u}{\vt{\tau}^2}+c}{\frac{v}{\vt{\tau}^2}} = \frac{c\vt{\tau}^2-bu+a}{v} = [c,-b,a]_{\tau} \\
&= \left(Q\circ S\right)_{\tau},
\end{align*}
and
\begin{align*}
Q_{T\tau} &= \frac{a\vt{\tau+1}^2+b\re(\tau+1)+c}{\im(\tau+1)} = \frac{a\left((u+1)^2+v^2\right)+b(u+1)+c}{v} \\
&= [a,2a+b,a+b+c]_{\tau} = \left(Q\circ T\right)_{\tau}.
\end{align*}
The second item is stated as a sentence directly in front of \cite{loeschw}*{Lemma 4.4}, and follows by \cite{brikako}*{(5.1), (7.12)}. The third item follows by a case by case analysis using the second item. Indeed, suppose that $\sgn(Q) = 1$. Then the second item implies that
\begin{align*}
\sgn\left(Q_{\tau}\right) = \begin{cases}
-1 & \text{if } \tau \in A_Q, \\
+1 & \text{if } \tau \not\in A_Q,
\end{cases}
\end{align*}
and this coincides with $\sgn(Q) \left(1-2\Id_Q(\tau)\right)$. The case $\sgn(Q) = -1$ follows in the same manner.
\end{proof}

\subsection{Cycle integrals}
Let $Q$ be such that $\Dc(Q)$ is positive and not a square. If $Q=[a,b,c]$ is primitive, and $t$, $r \in \N$ are the smallest solutions to Pell's equation $t^2-\Dc(Q)r^2 = 4$, the stabilizer $\Gamma_Q$ is generated by
\begin{align*}
\pm \left(\begin{array}{cc} \frac{t+br}{2} & cr \\ -ar & \frac{t-br}{2} \end{array}\right).
\end{align*}
If $Q$ is not primitive, one may divide its coefficients by $\gcd(a,b,c)$ to obtain a generator.

The weight $k$ cycle integral of a smooth function $h$, which transforms like a modular form of weight $k$, is defined as\footnote{The normalization by $\Dc(Q)^{\frac{1}{2}-\frac{k}{4}}$ is ommitted by some authors.}
\begin{align} \label{eq:cycdef}
\Cc_k(h,Q) \coloneqq \Dc(Q)^{\frac{1}{2}-\frac{k}{4}} \int_{\geo} h(z) Q(z,1)^{\frac{k}{2}-1} dz.
\end{align}
The integral is oriented counterclockwise if $\sgn(Q) > 0$, and clockwise if $\sgn(Q) < 0$. 

We collect the properties of cycle integrals in the following lemma, which can be proven by calculation, and the fact that $\Gamma_Q$ only depends on the equivalence class of $Q$.
\begin{lemma}
Let $f \colon \H \to \C$ be smooth, and suppose that $f$ is modular of weight $k$. Let $Q$ be a quadratic form of positive, non-square discriminant. Then the weight $k$ cycle integral $\Cc_{k}(f,Q)$ is a class invariant, namely it depends only on the equivalence class of $Q$ under $\sim$. Additionally, the weight $k$ cycle integral $\Cc_{k}(f,Q)$ is invariant under modular substitutions of the integration variable.
\end{lemma}
Hence, $\geo$ projects to a circle in a fundamental domain of $\Gamma$. The beautiful article \cite{duimto11} due to Duke, Imamo\={g}lu, T\'{o}th provides a good reference on Heegner geodesics as well as on cycle integrals.

\subsection{Maa{\ss} forms and modular forms}
We recall the definition of various classes of Maa{\ss} forms appearing in this paper. The slash operator is given by
\begin{align*}
\left(f\vert_k \left(\begin{smallmatrix} a & b \\ c & d \end{smallmatrix}\right)\right)(\tau) \coloneqq \begin{cases}
(c\tau+d)^{-k} f(\gamma\tau) & \text{if} \ k \in \Z, \\
\left(\frac{c}{d}\right)\varepsilon_d^{2k}(c\tau+d)^{-k} f(\gamma\tau) & \text{if} \ k \in \frac{1}{2}+\Z,
\end{cases}
\end{align*}
where $\left(\frac{c}{d}\right)$ denotes the Kronecker symbol, and
\begin{align*}
\varepsilon_d \coloneqq \begin{cases}
1 & \text{if} \ d \equiv 1 \pmod*{4}, \\
i & \text{if} \ d \equiv 3 \pmod*{4}. \\
\end{cases}
\end{align*}

\begin{defn}
Let $k \in \frac{1}{2}\Z$, choose $N \in \N$ such that $4 \mid N $ whenever $k \not\in \Z$, and let $f \colon \H \to \C$ be smooth.
\begin{enumerate}[label=(\roman*), labelindent=0pt, leftmargin=*]
\item We say that $f$ is a weight $k$ harmonic Maa{\ss} form for $\Gamma_0(N)$, if $f$ satisfies the following three properties:
\begin{enumerate}[label=(\alph*), leftmargin=*]
\item For all $\gamma \in \Gamma_0(N)$ and all $\tau \in \H$ we have $\left(f\vert_k\gamma\right)(\tau) = f(\tau)$.
\item The function $f$ is harmonic with respect to the weight $k$ hyperbolic Laplacian on $\H$, that is
\begin{align*}
0 = \Delta_k f \coloneqq \left(-v^2\left(\frac{\partial^2}{\partial u^2}+\frac{\partial^2}{\partial v^2}\right) + ikv\left(\frac{\partial}{\partial u} + i\frac{\partial}{\partial v}\right)\right) f.
\end{align*}
\item The function $f$ is of at most linear exponential growth towards all cusps of $\Gamma_0(N)$.
\end{enumerate}
\item A polar harmonic Maa{\ss} form is a harmonic Maa{\ss} form, which is permitted to posses isolated poles on the upper half plane.
\item A weak Maa{\ss} form satisfies conditions (a) and (c) of a harmonic Maa{\ss} form, but is allowed to have an arbitrary eigenvalue under $\Delta_k$.
\end{enumerate}
\end{defn}

To study his forms \cite{maass49}, Hans Maa{\ss} introduced the Maa{\ss} lowering and raising operators\footnote{Be aware that some authors shift their dependence on $k$, such as Maa{\ss} himself.} \cite{maass52}
\begin{align*}
L_k \coloneqq -2iv^2\frac{\partial}{\partial \overline{\tau}} = iv^2\left(\frac{\partial}{\partial u} + i\frac{\partial}{\partial v}\right), \quad R_k \coloneqq 2i \frac{\partial}{\partial \tau} + \frac{k}{v},
\end{align*}
which decreases or increases the weight of a weak Maa{\ss} form by $2$, and increases the eigenvalue under the hyperbolic Laplace operator by $2-k$ or $k$ respectively. A proof can be found in \cite{thebook}*{Lemma 5.2} for instance. For any $n \in \N_0$, we let 
\begin{align*}
L_{\kappa}^0 &\coloneqq \mathrm{id}, \quad L_{\kappa}^n \coloneqq L_{\kappa-2n+2} \circ \ldots \circ L_{\kappa-2} \circ L_{\kappa} \nonumber, \\
R_{\kappa}^0 &\coloneqq \mathrm{id}, \quad R_{\kappa}^n \coloneqq R_{\kappa+2n-2} \circ \ldots \circ R_{\kappa+2} \circ R_{\kappa}
\end{align*}
be the iterated Maa{\ss} lowering and raising operators respectively.

Bruinier, Funke \cite{brufu} introduced the shadow operator
\begin{align*}
\xi_k \coloneqq 2iv^k\overline{\frac{\partial}{\partial\overline{\tau}}} = iv^{k}\overline{\left(\frac{\partial}{\partial u} + i\frac{\partial}{\partial v}\right)}
\end{align*}
to study harmonic Maa{\ss} forms. They proved that the Fourier expansion of a harmonic Maa{\ss} form splits naturally into a holomorphic and a non-holomorphic part.

We define $M_k^!$ as the space of weakly holomorphic modular forms of weight $k$, and it turns out that $M_k^!$ is precisely kernel of $\xi_k$ restricted to weight $k$ harmonic Maa{\ss} forms. Analogously, a meromorphic modular form of weight $k$ can be regarded as an element of the kernel of $\xi_k$ restricted to weight $k$ polar harmonic Maa{\ss} forms. The space of holomorphic modular forms of weight $k$ is denoted by $M_k \subseteq M_k^!$. More details on various Maa{\ss} forms and their properties can be found in \cite{thebook} for instance.

\subsection{Poincar{\'e} series}
A first class of examples of (weakly) holomorphic modular forms, and of Maa{\ss} forms is given by constructing suitable Poincar{\'e} series. Such functions arise by averaging a specific auxiliary function (``seed''). Various seeds then lead to various examples of Poincar{\'e} series.

\begin{defn} \label{def:poincare}
\
\begin{enumerate}[label=(\roman*), leftmargin=*, labelindent=0pt]
\item For any $m \in \Z$, and any $\kappa \in \N_{>2}$, let 
\begin{align*}
P_{\kappa,m}(\tau) \coloneqq \sum_{\gamma \in \gaminf} q^m \big{\vert}_{\kappa}{\gamma}.
\end{align*}
be the weight $\kappa$ Poincar{\'e} series of exponential type.

\item Let $M_{\mu,\nu}$ be the usual $M$-Whittaker function, $m \in \Z\setminus\{0\}$, and define the seed
\begin{align*}
g_{m}(\tau,s) \coloneqq \frac{\Gamma(s)}{\Gamma(2s)}M_{0,s-\frac{1}{2}}(4 \pi \vt{m} y)\mathrm{e}^{2 \pi i m u}
\end{align*}
Then the Niebur Poincar{\'e} series \cites{nie73, neu73} is given by
\begin{align*}
G_m(\tau,s) \coloneqq \sum_{\gamma \in \gaminf} g_{m}(\tau,s)\big\vert_0 \gamma, \qquad \re(s) > 1.
\end{align*}
\item More generally, define the seed
\begin{align*}
\varphi_{\kappa,m}(\tau) \coloneqq \frac{(-\sgn(m))^{1-\kappa} (4 \pi \vt{m} v)^{-\frac{\kappa}{2}}}{\Gamma(2-\kappa)} M_{\sgn(m\kappa)\frac{\kappa}{2},\frac{1-\kappa}{2}}(4 \pi \vt{m} v) \mathrm{e}^{2 \pi \sgn(\kappa)m u}
\end{align*}
for any $m \in \Z\setminus\{0\}$, and $\kappa \in -\frac{1}{2}\N$. We require the Maa{\ss}-Poincar{\'e} series of negative integral weight $\kappa \in -\N$, which are defined as
\begin{align*}
\Phi_{\kappa,m}(\tau) \coloneqq \sum_{\gamma \in \gaminf} \varphi_{\kappa,m}(\tau)\big\vert_{\kappa} \gamma.
\end{align*}
\item We encounter one of Petersson's Poincar{\'e} series \cite{petersson50}, namely let $\cdot\vert_{k,z_1}$ be the weight $k$-operator acting on $z_1$, and let $k \in \N_{>2}$. Then we define
\begin{align*}
\Pb_k(z_1,z_2) &\coloneqq \im(z_2)^{k-1}\sum_{\gamma \in \Gamma} \left(\frac{1}{(z_1-z_2)(z_1-\overline{z_2})^{k-1}}\right)\Big\vert_{k,z_1} \gamma \\
&= \im(z_2)^{k-1}\sum_{\gamma \in \Gamma} \left(\frac{1}{(z_1-z_2)(z_1-\overline{z_2})^{k-1}}\right)\Big\vert_{2-k,z_2} \gamma
\end{align*}
\end{enumerate}
\end{defn}

We summarize their properties.
\begin{lemma} \label{lem:poincare}
\
\begin{enumerate}[label=(\roman*), leftmargin=*, labelindent=0pt] 
\item The function $P_{k,m}$ is a holomorphic cusp form for any $m > 0$, and a weakly holomorphic modular form for any $m < 0$.
\item The function $G_m(\tau,s)$ is a weak Maa{\ss} form of weight $0$ and eigenvalue $s(1-s)$ in $\tau$.
\item The function $\Phi_{\kappa,m}(\tau)$ is a harmonic Maa{\ss} form of weight $\kappa$. It decays like a cusp form towards all cusps inequivalent to $i\infty$, and the principal part at the cusp $i\infty$ is given by 
\begin{align*}
\varphi_{\kappa,m}(\tau)q^m.
\end{align*}
\item The function $\Pb_k(z_1,z_2)$ is a polar harmonic Maa{\ss} form of weight $2-k$ in $z_2$, and a meromorphic modular form of weight $k$ without a pole at the cusp in $z_1$. Moreover, the singularities of $\Pb_k(z_1,z_2)$ as a function of either argument are the $\Gamma$-orbits of the other argument.
\end{enumerate}
\end{lemma}

\begin{proof}
To check the claimed growth conditions, one has to compute the Fourier expansions and investigate the constant term in each expansion. We provide a reference for each item.
\begin{enumerate}[label=(\roman*), leftmargin=*, labelindent=0pt]
\item Compare with \cite{thebook}*{Theorems 6.8, 6.9}.
\item This is computed in \cite{fay77}*{Theorem 3.4} (see \cite{goldfield}*{eq.\ (1.13)}, \cite{duimto11}*{p.\ 19} as well).
\item This can be found in \cite{thebook}*{pp.\ 97}. The projection to Kohnen's plus space was calculated in \cite{bron07}*{Theorem 2.1}.
\item The statement in $z_1$ is due to Petersson \cite{petersson50}, see \cite{brika20}*{Proposition 3.3} as well. The statement in $z_2$ is proven in \cite{brika20}*{Proposition 3.2}. The claim dealing with the singularities of $\Pb_k$ follows by its definition.
\end{enumerate}
Modularity is obvious, and the analycicity condition is straightforward to check due to absolute convergence of each series. 
\end{proof}

We refer the reader to the exposition in \cite{brika20} for more details on $\Pb_k$ and related functions.

\subsection{Locally harmonic Maa{\ss} forms and local cusp forms} \label{sec:lhmf}
In \cite{brikako}, Bringmann, Kane, Kohnen introduced locally harmonic Maa{\ss} forms for $k > 1$, which were independently investigated for $k=1$ (i.\ e.\ weight $0$) by Hövel \cite{hoevel} in his PhD thesis as well. We follow \cite{brikako} here.
\begin{defn}
A locally harmonic Maa{\ss} form of weight $k$ for $\Gamma$ with exceptional set $X \subsetneq \H$ is a function $f \colon \H \to \C$, which satisfies the following properties:
\begin{enumerate}[label=(\roman*), leftmargin=*, labelindent=0pt]
\item For all $\gamma \in \Gamma$ and all $\tau \in \H$ we have $\left(f\vert_k\gamma\right)(\tau) = f(\tau)$.
\item For every $\tau \in \H \setminus X$, there exists a neighborhood of $\tau$, in which $f$ is real analytic and $\Delta_k f = 0$.
\item For every $\tau \in X$, we have
\begin{align*}
f(\tau) = \frac{1}{2} \lim_{\varepsilon \searrow 0} \left(f(\tau+i\varepsilon) + f(\tau-i\varepsilon)\right).
\end{align*}
\item The function $f$ exhibits at most polynomial growth towards the cusp $i\infty$, namely $f \in O\left(v^{\delta}\right)$ for some $\delta > 0$.
\end{enumerate}
\end{defn}
The points in the exceptional set $X$ are called ``jump singularities'' due to a wall-crosing behaviour between any two connected components of $\H \setminus X$. This definition is motivated by the peculiar first example
\begin{align*}
\Fc_{1-k,D}(\tau) \coloneqq \frac{(-1)^k D^{\frac{1}{2}-k}}{\binom{2k-2}{k-1} \pi} \sum_{Q \in \Qc(D)} \sgn\left(Q_{\tau}\right) Q(\tau,1)^{k-1} \psi_k\left(\frac{Dv^2}{\vt{Q(\tau,1)}^2}\right),
\end{align*}
where $D > 0$ is a non-square discriminant, and
\begin{align*}
\psi_k(y) \coloneqq \frac{1}{2} \int_0^y t^{k-\frac{3}{2}} (1-t)^{-\frac{1}{2}} dt
\end{align*}
is a special value of the incomplete $\beta$-function. We observe that ``locality'' is caused precisely by the presence of the sign function in the definition of $\Fc_{1-k,D}$, and indeed Bringmann, Kane, Kohnen proved that $\Fc_{1-k,D}$ satisfies their definition with weight $2-2k \in -2\N$ and exeptional set $\Nc(D)$. 

\begin{defn}
A local cusp form of weight $k$ for $\Gamma$ with exceptional set $X \subsetneq \H$ is a function $f \colon \H \to \C$, which satisfies the following properties:
\begin{enumerate}[label=(\roman*), leftmargin=*, labelindent=0pt]
\item For all $\gamma \in \Gamma$ and all $\tau \in \H$ we have $\left(f\vert_k\gamma\right)(\tau) = f(\tau)$.
\item For every $\tau \in \H \setminus X$, there exists a neighborhood of $\tau$, in which $f$ is holomorphic.
\item For every $\tau \in X$, we have
\begin{align*}
f(\tau) = \frac{1}{2} \lim_{\varepsilon \searrow 0} \left(f(\tau+i\varepsilon) + f(\tau-i\varepsilon)\right).
\end{align*}
\item The function $f$ vanishes as $\tau \to i\infty$.
\end{enumerate}
\end{defn}

Altogether, this motivates the definition and inspection of $\Ecwh_{k,D}(\tau,s)$.

\subsection{The functions $E_2^*$, $j$, and $j_m$}
The holomorphic Eisenstein series are given by
\begin{align*}
E_k(\tau) \coloneqq P_{k,0}(\tau) = 1-\frac{2}{\zeta(1-k)}\sum_{n\geq 1} \left(\sum_{\ell \mid n} \ell^{k-1} \right) q^n,
\end{align*}
where $\zeta$ denotes the Riemann zeta function. If $k \geq 4$ is even then $E_k \in M_k(\Gamma)$, and $E_2$ is quasimodular. We define
\begin{align*}
E_2^*(\tau) \coloneqq E_2(\tau) - \frac{3}{\pi v},
\end{align*}
and observe that $E_2^*$ is a harmonic Maa{\ss} form of weight $2$ for $\Gamma$ (cf.\ \cite{thebook}*{Lemma 6.2}). The modular invariant for $\Gamma$ is the function
\begin{align*}
j(\tau) \coloneqq \frac{E_4(\tau)^3}{\Delta(\tau)} \in M_0^!(\Gamma),
\end{align*}
where
\begin{align*}
\Delta(\tau) \coloneqq q\prod_{n \geq 1}\left(1-q^n\right)^{24} = \frac{E_4(\tau)^3-E_6(\tau)^2}{1728} \in S_{12}(\Gamma)
\end{align*}
is the normalized modular discriminant function. We denote the normalized derivative of $j$ by 
\begin{align*}
j'(\tau) \coloneqq \frac{1}{2\pi i} \frac{\partial j}{\partial \tau}(\tau) = -\frac{E_4(\tau)^2 E_6(\tau)}{\Delta(\tau)} \in M_2^!(\Gamma).
\end{align*}
The latter identity can be verified by Ramanujan's differential system \cite{the123}*{Proposition 15}
\begin{align*}
\frac{1}{2\pi i} \frac{\partial E_2}{\partial \tau} = \frac{E_2^2-E_4}{12}, \qquad \frac{1}{2\pi i} \frac{\partial E_4}{\partial \tau} = \frac{E_2E_4-E_6}{3}, \qquad \frac{1}{2\pi i} \frac{\partial E_6}{\partial \tau} = \frac{E_2E_6-E_4^2}{2}.
\end{align*}
As an intermediate result, one can check that 
\begin{align*}
\frac{1}{2\pi i} \frac{\partial \Delta}{\partial \tau} = E_2(\tau)\Delta(\tau).
\end{align*}
For every $m \geq 0$, let $j_m(\tau)$ be the unique function in the space $M_0^!(\Gamma)$ having a Fourier expansion of the form $q^{-m} + O(q)$. For instance, we have
\begin{align*}
j_0(\tau) = 1, \qquad j_1(\tau) = j(\tau) - 744, \qquad j_2(\tau) = j(\tau)^2 - 1488j(\tau) + 159768,
\end{align*}
and the set $\left\{j_m \colon m \geq 0\right\}$ is a basis for $M_0^!$. This was proven by Asai, Kaneko, Ninomiya \cite{askani}, and they additionally established the expansion
\begin{align*}
\frac{j'(\tau)}{j(w)-j(\tau)} = \sum_{m \geq 0} j_m(w) q^m, \qquad \im{(\tau)} > \im{(w)}.
\end{align*}
Alternatively, the functions $j_m$ can be constructed following \cite{brikaloeonro}. More precisely, the authors proved that the functions $j_m$ form a Hecke system, that is if $T_m$ denotes the normalized Hecke operator, then define $j_0$, $j_1$ as above, and extend inductively by
\begin{align*}
j_m = T_mj_1.
\end{align*}

\section{Hyperbolic Eisenstein series} \label{sec:hypeis}
Let $D > 0$ be a non-square discriminant, $d$ be a positive fundamental discriminant dividing $D$, and $k \in 2\N$. We recall the definition of our two hyperbolic Eisenstein series from the introduction (see equations \eqref{eq:eisold}, \eqref{eq:eisnew}), and the fact that both converge absolutely for any $s \in \C$ with $\re(s) > 1-\frac{k}{2}$.

\begin{rmk}
Let $d_{\mathrm{hyp}}$ be the hyperbolic distance. Then, we have
\begin{align*}
\frac{\vt{Q(\tau,1)}}{v} = \Dc(Q)^{\frac{1}{2}}\cosh{\left(d_{\mathrm{hyp}}(\tau,S_Q)\right)}.
\end{align*}
A proof of this idendity can be found in \cite{voelz18}*{Lemma 2.5.4}. Note that $z \in S_Q$ if and only if $d_{\mathrm{hyp}}\left(z,S_Q\right) = 0$.
\end{rmk}

As outlined in the introduction, the function $\Ec_{2,D}$ possesses an analytic continuation to $s=0$, which can be proven by computing the Fourier expansion of $\Ec_{2,D}$. We recall the result for convenience.
\begin{lemma}[\protect{\cite{hypeis1}*{Theorem 1.1}}] \label{lem:old}
Let $D > 0$ be a non-quare discriminant, $d$ be a positive fundamental discriminant dividing $D$. Then the function $\Ec_{2,D}(\tau,s)$ can be analytically continued to $s=0$ and the continuation is given by
\begin{align*}
\lim_{s \to 0} \Ec_{2,D}(\tau,s) = \frac{-2}{D} \sum_{m \geq 0} \sum_{Q \in \Qc(D)\slash\Gamma} \chi_d(Q) \Cc_0\left(j_m(\cdot) - E_2^*(\tau),Q\right) q^m \\
\end{align*}
for any $\tau \in \H$. Furthermore, if $v$ is sufficiently large, that is $\tau$ is located above the net of geodesics $\Nc(D)$, then we have
\begin{align*}
\lim_{s \to 0} \Ec_{2,D}(\tau,s) = \frac{-2}{D} \sum_{Q \in \Qc(D)\slash\Gamma} \chi_d(Q) \Cc_0\left(\frac{j'(\tau)}{j(w)-j(\tau)} - E_2^*(\tau),Q\right).
\end{align*}
\end{lemma}

Along the lines of Lemma \ref{lem:signs} (iii), we define
\begin{align*}
\Ecwt_{k,D}(\tau,s) \coloneqq \sum_{Q \in \Qc(D)\slash\Gamma} \chi_d\left(Q\right) \sum_{\widehat{Q} \sim Q} \frac{\sgn{(\widehat{Q})}^{\frac{k}{2}} \Id_Q(\tau) v^s}{\widehat{Q}(\tau,1)^{\frac{k}{2}} \vt{\widehat{Q}(\tau,1)}^s}.
\end{align*}

\begin{rmk}
In \cite{zagier10}, Zagier introduces the notion of quantum modular forms, and discusses some examples. In particular, his second example involves the quantum modular form
\begin{align*}
\sum_{Q \in \Qc(D)} \max\{Q(x,1),0\}^5 = \sum_{\substack{Q \in \Qc(D) \\ Q(x,1) > 0}} Q(x,1)^5, \qquad x \in \Q,
\end{align*}
which appears also in his paper \cite{zagier99}. By Lemma \ref{lem:signs} (ii), we have $\Id_Q(\tau) = 1$ if and only if $\sgn(Q)\sgn(vQ_{\tau}) = -1$. As the zeros of $Q(\tau,1)$ are quadratic irrationals, the limit $\lim_{\tau \to x} \frac{1}{Q(\tau,1)}$ exists for every $x \in \Q$. Furthermore, we note that 
\begin{align*}
\lim_{\tau \to x} (vQ_{\tau}) = \lim_{\tau \to x} \left(a\vt{\tau}^2+bu+c\right) = Q(x,1).
\end{align*}
Altogether, this suggests that there might be a connection of the rational function (taking $d=1$ here)
\begin{align*}
x \mapsto \lim_{\tau \to x} \Ecwt_{2k-2,D}(\tau,0) = \sum_{\substack{Q \in \Qc(D) \\ \sgn(Q)\sgn(Q(x,1))=-1}} \frac{\sgn(Q)^{k-1}}{Q(x,1)^{k-1}} = -2\sum_{\substack{Q \in \Qc(D) \\ Q(x,1) > 0 }} \frac{1}{Q(x,1)^{k-1}}
\end{align*}
to quantum modular forms for certain weights $k$.
\end{rmk}

We combine Lemmas \ref{lem:signs}, \ref{lem:old}.
\begin{prop}
Assume that $0 < k \equiv 2 \pmod*{4}$, $\tau \in \H\setminus\Nc(D)$, and $\re(s) > 1-\frac{k}{2}$. 
\begin{enumerate}[label=(\roman*), leftmargin=*, labelindent=0pt]
\item The function $\Ecwh_{k,D}(\tau, s)$ is modular of weight $k$, and we have
\begin{align}
\Ecwh_{k,D}(\tau, s) = \Ec_{k,D}(\tau, s) - 2 \Ecwt_{k,D}(\tau,s). \label{eq:eisidentity}
\end{align}
\item The function $\Ecwh_{2,D}(\tau, s)$ has an analytic continuation to $s=0$.
\item The identity \eqref{eq:eisidentity} holds for the case of $k=2$, $s=0$ as well.
\end{enumerate}
\end{prop}

\begin{proof}
The first item is a direct consequence of Lemma \ref{lem:signs}. Thus, it suffices to show that $\Ecwt_{2,D}(\tau,s)$ has an analytic continuation to $s=0$ to prove the second item. To this end, we observe that $\Ecwt_{k,D}(\tau,s)$ vanishes above the net of geodesics $\Nc(D)$, and coincides locally with $\Ec_{k,D}(\tau,s)$ up to some non-zero constant in any bounded connected component of $\H \setminus \Nc(D)$. Hence, one may obtain a Fourier expansion of $\Ecwt_{2,D}$ locally by Lemma \ref{lem:old}. (The computation was presented in \cite{hypeis1}). This establishes the existence of $\Ecwh_{2,D}(\tau,0)$ via the identity \eqref{eq:eisidentity} from the first item, and in addition proves the third item by uniqueness of the limit.
\end{proof}

Moreover, we recall the Fourier expansion of $\Ec_{k,D}(\tau,0)$ for higher weights.
\begin{lemma}[\protect{\cite{hypeis1}*{Theorem 1.2}}]
Let $D > 0$ be a non-square discriminant, let $d$ be a positive fundamental discriminant dividing $D$, and suppose that $k \geq 4$ is even. Then, we have the Fourier expansion
\begin{align*}
\Ec_{k,D}(\tau,0) = \frac{(-1)^{\frac{k}{2}}2 \pi^{\frac{k}{2}}}{D^{\frac{k+2}{4}}\Gamma\left(\frac{k}{4}\right)^2} \sum_{m \geq 1} m^{\frac{k}{2}-1} \sum_{Q \in \Qc(D)\slash\Gamma} \chi_d(Q) \Cc_0 \left(G_{-m}\left(\cdot,\frac{k}{2}\right)\right) q^m.
\end{align*}
\end{lemma}

Since $\Ec_{k,D}$ converges absolutely on $\H$ at $s=0$ for any $k \geq 4$ even, we may rearrange its Fourier expansion, and study the integrand
\begin{align*}
f(w,\tau) \coloneqq \sum_{m \geq 1} m^{\frac{k}{2}-1} G_{-m}\left(w,\frac{k}{2}\right) q^m, \quad w \in \geo, \quad \tau \in \H
\end{align*}
inside the cycle integral. In other words, we may rewrite the Fourier expansion from the previous Lemma as
\begin{align*}
\Ec_{k,D}(\tau, 0) = \frac{(-1)^{\frac{k}{2}}2 \pi^{\frac{k}{2}}}{D^{\frac{k+2}{4}}\Gamma\left(\frac{k}{4}\right)^2} \sum_{Q \in \Qc(D)\slash\Gamma} \chi_d(Q) \Cc_0(f(\cdot,\tau),Q).
\end{align*}
We obtained an alternative representation of the Fourier expansion of $\Ec_{2,D}(\tau, 0)$ already if $\tau$ is located in the unbounded component of a fundamental domain for $\Gamma$. The main ingredient to prove the second claim of Theorem \ref{thm:main} is to find such an representation in the case of higher weights under the same assumption on $\tau$.
\begin{prop} \label{prop:cycpoin}
Let $2 < k \equiv 2 \pmod*{4}$, let $D > 0$ be a non-square discriminant, and $d$ be a positive fundamental discriminant dividing $D$. Suppose that $v$ is sufficiently large, that is $\tau$ is located above the net of geodesics $\Nc(D)$. Then $\Ec_{k,D}(\tau,0)$ coincides with the function
\begin{align*}
\sum_{Q \in \Qc(D)\slash\Gamma} \chi_d(Q) \Cc_{2-k}\left(\Pb_k(\tau,\cdot),Q\right)
\end{align*}
up to an explicit non-zero constant, which is provided in equation \eqref{eq:finalconstant}.
\end{prop}

\begin{rmk}[Rearranging the Fourier expansion]
Let $W_{\mu,\nu}$ be the usual $W$-Whittaker function. Inserting the Fourier expansion of $G_{-m}$, next comparing with the Fourier expansion of $P_{k,m}$ (see the proof of Lemma \ref{lem:poincare} for a list of references), and rearranging further, one obtains
\begin{multline*}
f(w,\tau) = \frac{\Gamma\left(\frac{k}{2}\right)}{\Gamma(k)} \sum_{m \geq 1} m^{\frac{k}{2}-1} M_{0,\frac{k}{2}-\frac{1}{2}}(4 \pi \vt{m} \im(w)) \mathrm{e}^{-2\pi i m \re(w)} q^m \\
 + \frac{2^{2-k}\pi^{-\frac{k}{2}}\Gamma(k)}{(k-1)\Gamma\left(\frac{k}{2}\right)}  \sin\left(\frac{\pi}{2}(1-k)\right) \im(w)^{1-\frac{k}{2}} \left(E_k(\tau)-1\right) \\
 + i^{-k} \sum_{n \neq 0} \vt{n}^{\frac{k-1}{2}} W_{0,\frac{k}{2}-\frac{1}{2}}(4\pi \vt{n} \im(w)) \left(P_{k,n}(\tau)-q^{n}\right)\mathrm{e}^{-2\pi i  n \re(w)}.
\end{multline*}
However, we may not split the final sum involving $P_{k,n}(\tau)-q^{n}$ into two separate sums over $n$, since the resulting expressions would not converge with respect to $\tau$. This emphasizes the error to modularity of $\Ec_{k,D}$ from a different viewpoint.
\end{rmk}

\section{Proof of Theorem \ref{thm:main}} \label{sec:proofmain}
We begin with the proof of Proposition \ref{prop:cycpoin}. To this end, we write $w = x+iy \in \geo$ for the integration variable of the cycle integral, and collect three intermediate results first. In case of ambiguity, we specify the variable a Maa{\ss} operator shall act on by an additional subscript next to the weight.

The first step is to convert $G_{-m}$ to a harmonic Maa{\ss} form.
\begin{lemma}
We have
\begin{align*}
\left(L_0^{\frac{k}{2}-1} G_{-m}\right)\left(w,\frac{k}{2}\right) = \frac{C_1(k)\Gamma(k)}{(8\pi\vt{m})^{\frac{k}{2}-1}} \Phi_{2-k,-m}(w), \quad C_1(k) \coloneqq \prod_{j=0}^{\frac{k-4}{2}} (k+2j).
\end{align*}
\end{lemma}

\begin{proof}
By absolute convergence, we may differentiate the seed directly. We calculate that
\begin{align*}
L_0^{\frac{\ell}{2}+1} \left(M_{0,\frac{k}{2}-\frac{1}{2}}(4 \pi \vt{m} y)\mathrm{e}^{-2 \pi i m x}\right) = \prod_{j=0}^{\frac{\ell}{2}} (k+2j) \left(\frac{y}{2}\right)^{\frac{\ell}{2}+1} M_{\frac{\ell}{2}+1,\frac{k}{2}-\frac{1}{2}}(4 \pi \vt{m} y)\mathrm{e}^{-2 \pi i m x}
\end{align*}
for every $\ell \in 2\N_0$. We compare this with the definition of the seed $\varphi_{\kappa,m}$, and choose $\ell = k-4$. This yields the claim.
\end{proof}

The second step is to connect this result to the Fourier expansion of $\Ec_{k,D}(\tau,0)$. Thus, we need an identity involving (iterated) Maa{\ss} operators and cycle integrals. This was performed by Alfes-Neumann, Schwagenscheidt \cite{alneschw}, generalizing earlier results of Bringmann, Guerzhoy, Kane \cites{brguka14, brguka15}. To simplify the notation, we drop the weights of the cycle integrals temporarily.
\begin{lemma}[\protect{\cite{alneschw}*{Theorem 1.1}}] \label{lem:cycid}
Let $h \colon \H \to \C$ be a smooth function, which transforms like a modular form of weight $2-2\kappa \in 2\Z$ for $\Gamma$. Then we have the identity
\begin{align*}
\Cc(L_{2-2\kappa}h,Q) = \Cc(R_{2-2\kappa}h,Q) = \overline{\Cc(\xi_{2-2\kappa}h,Q)}.
\end{align*}
Moreover, if $h$ is a weak Maa{\ss} form of weight $2-2\kappa$ with eigenvalue $\lambda$, then we have
\begin{align}
\Cc\left(R_{2-2\kappa}^{\kappa-\ell}h,Q\right) &= \left((\kappa+\ell)(\kappa-\ell-1)-\lambda\right)\Cc\left(R_{2-2\kappa}^{\kappa-\ell-2}h,Q\right), \text{ if } \ell \leq \kappa-2, \label{eq:cycraise} \\
\Cc\left(L_{2-2\kappa}^{-\kappa-\ell+2}h,Q\right) &= \left((\kappa+\ell)(\kappa-\ell-1)-\lambda\right)\Cc\left(L_{2-2\kappa}^{-\kappa-\ell}h,Q\right), \text{ if } \ell \leq -\kappa. \label{eq:cyclow}
\end{align}
\end{lemma}
Note that the conditions on $\ell$ in \eqref{eq:cycraise}, \eqref{eq:cyclow} include the cases $R_{2-2\kappa}^0$, $L_{2-2\kappa}^0$. Thus, we may insert a suitable chain of raising or lowering operators in our cycle integrals and compensate for that by factors in $\kappa$, $\ell$ from equations \eqref{eq:cycraise}, \eqref{eq:cyclow}.

The third step is to utilize an identity due to Bringmann, Kane \cite{brika20}.
\begin{lemma}[\protect{\cite{brika20}*{eq.\ (3.10), (3.11)}}] \label{lem:brikaid}
We have
\begin{align*}
\sum_{m \geq 1} \Phi_{2-k,-m}(w)q^m = \frac{i}{2\pi} (2i)^{k-1} \Pb_k(\tau,w),
\end{align*}
whenever
\begin{align*}
\im(\tau) > \max\left(\im(w), \frac{1}{\im(w)}\right).
\end{align*}
\end{lemma}

Now, we are in position to prove Proposition \ref{prop:cycpoin}.
\begin{proof}[Proof of Proposition \ref{prop:cycpoin}]
Since $\tau$ is assumed to be located above the net of geodesics, the assumption from Lemma \ref{lem:brikaid} is satisfied for every $w \in \Nc(D)$. ($\im(w)$ is bounded from below and above.) In addition, we have no poles of $\Pb_k$ for such $\tau$ and $w$.

We invoke Lemma \ref{lem:cycid}, and employ equation \eqref{eq:cyclow} backwards and iteratively to the integrand
\begin{align*}
f(w,\tau) = \sum_{m \geq 1} m^{\frac{k}{2}-1} G_{-m}\left(w,\frac{k}{2}\right) q^m,
\end{align*}
from the Fourier expansion of $\Ec_{k,D}$. Here, we keep $\tau$ fixed, and take $\kappa = 1$, $\lambda = \frac{k}{2}\left(1-\frac{k}{2}\right)$, and $\ell = -1,-3,\ldots,-\frac{k}{2}+2$ using that $k \equiv 2 \pmod*{4}$. This produces the constant
\begin{align*}
C_2(k) \coloneqq \prod_{\substack{\ell = -\frac{k}{2}+2 \\ \ell \text{ odd}}}^{-1} \frac{1}{(1+\ell)(-\ell)-\frac{k}{2}\left(1-\frac{k}{2}\right)}.
\end{align*}
To indicate the steps, we keep the constants until the last equation. Combining, we have
\begin{align*}
&\Ec_{k,D}(\tau, 0) = \frac{(-1)^{\frac{k}{2}}2 \pi^{\frac{k}{2}}}{D^{\frac{k}{4}}\Gamma\left(\frac{k}{4}\right)^2} \sum_{Q \in \Qc(D)\slash\Gamma} \chi_d(Q) \Cc_0\left(L_{0,\cdot}^0f(\cdot,\tau),Q\right) \\
&= \frac{(-1)^{\frac{k}{2}}2 \pi^{\frac{k}{2}} C_2(k)}{D^{\frac{k}{4}}\Gamma\left(\frac{k}{4}\right)^2} \sum_{Q \in \Qc(D)\slash\Gamma} \chi_d(Q) \Cc_{2-k}\left(L_{0,\cdot}^{\frac{k}{2}-1}f(\cdot,\tau),Q\right) \\
&= \frac{(-1)^{\frac{k}{2}}2 \pi^{\frac{k}{2}} C_2(k)}{D^{\frac{k}{4}}\Gamma\left(\frac{k}{4}\right)^2} \frac{C_1(k)\Gamma(k)}{(8\pi)^{\frac{k}{2}-1}} \sum_{Q \in \Qc(D)\slash\Gamma} \chi_d(Q) \Cc_{2-k}\left(\sum_{m \geq 1} \Phi_{2-k,-m}(\cdot)q^m,Q\right) \\
&= \frac{(-1)^{\frac{k}{2}}2 \pi^{\frac{k}{2}} C_2(k)}{D^{\frac{k}{4}}\Gamma\left(\frac{k}{4}\right)^2} \frac{C_1(k)\Gamma(k)}{(8\pi)^{\frac{k}{2}-1}} \frac{i}{2\pi} (2i)^{k-1} \sum_{Q \in \Qc(D)\slash\Gamma} \chi_d(Q) \Cc_{2-k}\left(\Pb_k(\tau,\cdot),Q\right)
\end{align*}
The constant in front of the final sum simplifies to
\begin{align}
C(k) \coloneqq \frac{(-1)^k \Gamma(k)}{2^{\frac{k}{2}-2} D^{\frac{k}{4}}\Gamma\left(\frac{k}{4}\right)^2} C_1(k)C_2(k). \label{eq:finalconstant}
\end{align}
This establishes the Propostition.
\end{proof}

We conclude this section and the paper with the proof of Theorem \ref{thm:main}.
\begin{proof}[Proof of Theorem \ref{thm:main}]
\
\begin{enumerate}[label=(\roman*), leftmargin=*, labelindent=0pt]
\item The case $k=2$ was shown in \cite{hypeis1} in the unbounded component of $\H \setminus \Nc(D)$ for $\Ec_{2,D}(\tau,0)$. Since $\Ecwh_{k,D}(\tau,0) = \Ec_{k,D}(\tau,0)$ in the unbounded component by definition of $\Id_Q$, the result of \cite{hypeis1} extends to $\Ecwh_{k,D}(\tau,0)$ in the unbounded component of $\H \setminus \Nc(D)$ directly. Now, we can use modularity of $\Ecwh_{k,D}(\tau,0)$ to obtain the claim in the other connected components of $\H \setminus \Nc(D)$.
\item Suppose that $2 < k \equiv 2 \pmod{4}$. Modularity follows by Lemma \ref{lem:signs} (i). By Lemma \ref{lem:signs} (iii), $\Ecwh_{k,D}(\tau,0)$ is holomorphic outside $\Nc(D)$. The limit condition on $\Nc(D)$ can be verified by adapting the proof of \cite{brikako}*{Proposition $5.2$} straightforwardly. The vanishing at $i\infty$ either follows by $\sgn(Q_{\tau}) = 1$ in the unbounded component and cuspidality of $f_{k,D}$, or by the Fourier expansions of $\Ec_{k,D}(\tau,0)$ and $\Ecwt_{k,D}(\tau,0)$.
\item We prove the explicit representation of $\Ecwh_{k,D}(\tau,0)$ outside $\Nc(D)$. If $\tau$ is located above the net of geodesics $\Nc(D)$, we have $\Ecwh_{k,D}(\tau,0) = \Ec_{k,D}(\tau,0)$. We apply Propostion \ref{prop:cycpoin}, and obtain the claimed representation of $\Ecwh_{k,D}$ above the net of geodesics. Finally, the representation extends to every connected component of $\H \setminus \Nc(D)$ by virtue of weight $k$ modularity of both sides of the claimed identity.
\end{enumerate}
\end{proof}


\end{document}